\documentclass[11pt]{article}
\usepackage{amssymb,amsmath,amsthm,enumerate,cite,geometry,float,tikz,pgfplots}
\newtheorem{theorem}{Theorem}
\newtheorem{corollary}[theorem]{Corollary}
\usepackage{lineno}

\usepackage{setspace}
\allowdisplaybreaks

\let\leq\leqslant
\let\geq\geqslant
\let\epsilon\varepsilon
\let\emptyset\varnothing

\begin{document}
\onehalfspace
\title{On the number of maximal independent sets and \\
maximal induced bipartite subgraphs in $K_4$-free graphs}
\author{Thilo Hartel$^1$ \and Lucas Picasarri-Arrieta$^{2}$ \and  Dieter Rautenbach$^1$}
\date{}
\maketitle
\begin{center}
$^1$ Ulm University, Ulm, Germany\\
\texttt{$\{$thilo.hartel,dieter.rautenbach$\}$@uni-ulm.de}\\[3mm]
$^2$ National Institute of Informatics, Tokyo, Japan\\
\texttt{lpicasarr@nii.ac.jp}
\end{center}

\begin{abstract}
Let $G$ be a $K_4$-free graph of order $n$ and 
let $k$ be an integer with $0\leq k\leq n$.
We show the existence of positive constants $\eta$ and $\nu$ such that $G$ has 
at most $(4-\eta)^{(5-\eta)k-n}(5-\eta)^{n-(4-\eta)k}$
maximal independent sets of order $k$
and 
at most $O\left((12-\nu)^{\frac{n}{4}}\right)$
maximal induced bipartite subgraphs.\\[3mm]
{\bf Keywords}: Maximal independent set; maximal induced bipartite subgraphs
\end{abstract}

\section{Introduction}

We consider finite, simple, and undirected graphs and use standard terminology.
For a graph $G$ of order $n$ and a non-negative integer $k$, 
let ${\rm MIS}(G)$ be the set of all maximal independent sets of $G$,
let ${\rm mis}(G)=|{\rm MIS}(G)|$,
let ${\rm mis}_k(G)=|\{ I\in {\rm MIS}(G): |I|=k\}$, and
let ${\rm mis}(G)_{\leq k}=|\{ I\in {\rm MIS}(G): |I|\leq k\}$.
A famous result established independently 
by Miller and Muller \cite{mimu}
and 
by Moon and Moser \cite{momo} 
is that 
\begin{eqnarray}\label{e1}
{\rm mis}(G) & \leq & 3^{\frac{n}{3}}.
\end{eqnarray}
Eppstein \cite{ep} showed 
\begin{eqnarray}\label{e2}
{\rm mis}_{\leq k}(G) & \leq & 3^{4k-n}4^{n-3k}.
\end{eqnarray}
The bound \eqref{e1} satisfied with equality if $G$ is the disjoint union of triangles.
Following Eppstein's inductive proof,
we show the following.

\begin{theorem}[Eppstein, refined]\label{theorem2}
If $G$ is a graph of order $n$ and $k$ is a non-negative integer, then
${\rm mis}_{\leq k}(G)\leq 3^{4k-n}4^{n-3k}$
with equality if and only if 
$G$ is the disjoint union of $k$ copies of $K_3$ or $K_4$.
\end{theorem}

Our motivation for the present paper 
was an attempt to improve a result of Byskov, Madsen, and Skjernaa \cite{bymask}.
Improving earlier results of Schiermeyer \cite{sc},
they showed that every graph $G$ of order $n$
has at most $O\left(12^{\frac{n}{4}}\right)$
maximal induced bipartite subgraphs.
In fact, they show that every such subgraph $H$ 
has a bipartition $V(H)=A\cup B$ with $|A|\geq |B|$ such that 
$A$ is a maximal independent set of $G$ and 
$B$ is a maximal independent set of $G-A$.
Fixing such a bipartition for every maximal induced bipartite subgraphs $H$, 
counting these subgraphs according to the size $k$ of $A$,
and using both \eqref{e1} and \eqref{e2}
yields that their number is at most
\begin{eqnarray}\label{e3}
\sum\limits_{k=0}^{p}\,\,\, 
\underbrace{3^{4k-n}4^{n-3k}}_{\geq \# options\, for\, A}\,\,\cdot \,\,
\underbrace{3^{4k-(n-k)}4^{(n-k)-3k}}_{\geq \# options\, for\, B\, given\, A}
\,\,\,+\,\,\,
\sum\limits_{k=p+1}^{n}\,\,\, 
\underbrace{3^{4k-n}4^{n-3k}}_{\geq \# options\, for\, A}\,\,\cdot \,\,
\underbrace{3^{\frac{n-k}{3}},}_{\geq \# options\, for\, B\, given\, A}
\end{eqnarray}
where choosing $p=\left\lfloor\frac{n}{4}\right\rfloor$ 
yields the stated estimate $O\left(12^{\frac{n}{4}}\right)$.
In fact, the terms of the first sum increase with $k$
while the terms of the second sum decrease with $k$
and for $p=\left\lfloor\frac{n}{4}\right\rfloor$
the largest terms in both sums have the same order.
Therefore, trying to improve the result of Byskov et al.~\cite{bymask} 
motivates special attention to the case $k=\frac{n}{4}$.
By Theorem \ref{theorem2}, 
in this case the estimate \eqref{e2},
which is used twice in \eqref{e3} 
to bound the number of options for $A$,
is tight only if 
$G$ is the disjoint union of $k$ copies of $K_4$
and the argument of Byskov et al. \cite{bymask}
overcounts the maximal induced bipartite subgraphs.
The reason for this overcounting is that 
$K_4$ has $12$ ordered pairs $(A,B)$ of disjoint maximal independent sets
--- the $12$ ordered pairs of distinct vertices ---
but only $6$ maximal induced bipartite subgraphs
corresponding to its edges.

The above reasoning motivates to study upper bounds on 
${\rm mis}_k(G)$
for $K_4$-free graphs $G$ or, more generally, in $K_s$-free graphs.
Note that lower bounds for this setting were provided by He, Nie, and Spiro \cite{henisp}.
For $k=\frac{n}{4}$, Eppstein's bound \eqref{e2} simplifies to 
${\rm mis}_k(G) \leq 4^{\frac{n}{4}}$.
We improve this for $K_4$-free graphs and $k$ close to $\frac{n}{4}$.

\begin{theorem}\label{theorem1}
There exist $\epsilon, \delta>0$ with the following property.
If $G$ is a $K_4$-free graph of order $n$ and maximum degree at most $3$,
and $k\leq (1+\epsilon)\frac{n}{4}$ is a positive integer,
then 
\begin{eqnarray}\label{e4}
{\rm mis}_k(G) & \leq & (4-\delta)^{\frac{n}{4}}.
\end{eqnarray}
\end{theorem}
Theorem \ref{theorem1} allows to deduce the following bound valid for all values of $k$
and without restricting the maximum degree.

\begin{corollary}\label{corollary1}
There exists $\eta_0\in \left(0,1\right)$ 
such that, for every $\eta\in \left(0,\eta_0\right)$,
the following holds.
If $G$ is a $K_4$-free graph of order $n$ and $k$ is an integer with $0\leq k\leq n$,
then 
\begin{eqnarray}\label{e5}
{\rm mis}_k(G) & \leq & 
(4-\eta)^{(5-\eta)k-n}(5-\eta)^{n-(4-\eta)k}.
\end{eqnarray}
\end{corollary}
For $k$ close to $\frac{n}{4}$,
the bound \eqref{e5} improves Eppstein's bound \eqref{e2}
as well as a bound due to Nielsen \cite{ni}.
In fact, Nielsen's Theorem 2 in \cite{ni} implies
${\rm mis}_k(G) \leq 4^{5k-n}5^{n-4k}$
for every graph $G$ of order $n$ and every integer $k$ with $0\leq k\leq n$.
See Figure \ref{fig0} for an illustration.

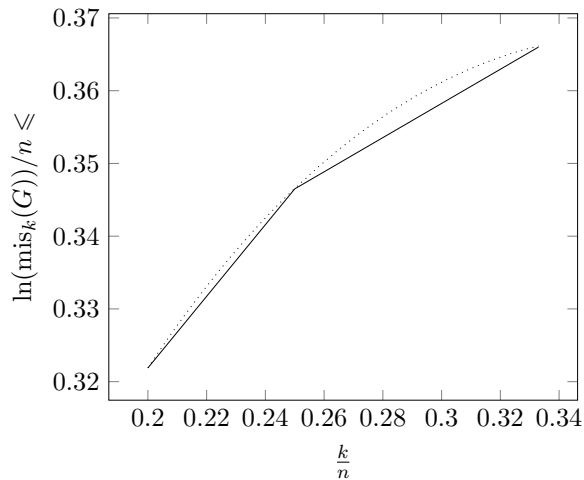
\begin{figure}[h]
    \centering
    \begin{tikzpicture}[scale=0.9]
    \begin{axis}[
    ylabel=$\ln({\rm mis}_k(G))/n\leq$,
    xlabel=$\frac{k}{n}$,
    samples=100,domain=0.2:0.333]
    \addplot[thin, dotted, mark=none ]plot (\x,\x*ln(1/\x);
    \addplot [thin] table {
0.2 0.32188
0.25 0.3465
0.333 0.366
};   
    \end{axis}
    \end{tikzpicture}
\caption{
The plot shows upper bounds on $\frac{\ln({\rm mis}_k(G))}{n}$ 
as a function of $x=\frac{k}{n}$ for the range $\frac{k}{n}\in \left[\frac{1}{5},\frac{1}{3}\right]$,
where $G$ is a graph of order $n$.
The straight line segment for $\frac{k}{n}\in \left[\frac{1}{4},\frac{1}{3}\right]$
is $\left(4\frac{k}{n}-1\right)\ln(3)+\left(1-3\frac{k}{n}\right)\ln(4)$, 
which corresponds to Eppstein's bound \eqref{e2}.
The straight line segment for $\frac{k}{n}\in \left[\frac{1}{5},\frac{1}{4}\right]$
is $\left(5\frac{k}{n}-1\right)\ln(4)+\left(1-4\frac{k}{n}\right)\ln(5)$, 
which corresponds to Nielsen's bound mentioned above.
The dotted curve is the smooth interpolation $x\mapsto x\ln\left(\frac{1}{x}\right)$,
which coincides with Nielsen's bound 
from \cite{ni}
whenever $\frac{1}{x}$ is an integer.
The bound \eqref{e5} from Corollary \ref{corollary1} corresponds to a line 
through the two points on the dashed curve for 
$x=\frac{1}{4-\eta}\in \left(\frac{1}{4},\frac{1}{3}\right)$
and 
$x=\frac{1}{5-\eta}\in \left(\frac{1}{5},\frac{1}{4}\right)$,
which slightly improves Eppstein's and Nielsen's bounds close to $x=\frac{1}{4}$.
See also the proof of Corollary \ref{corollary1}
and Figure \ref{figcor1}.}\label{fig0}
\end{figure}
Concerning the maximum number of maximal induced bipartite subgraphs,
our results allow the following.

\begin{corollary}\label{corollary2}
There exists $\nu>0$ with the following properties.
\begin{enumerate}[(i)]
\item If $G$ is a $K_4$-free graph of order $n$,
then $G$ has at most $O\left((12-\nu)^{\frac{n}{4}}\right)$
maximal induced bipartite subgraphs.
\item If $G$ is a graph of order $n$ and maximum degree at most $3$,
then $G$ has at most $O\left((12-\nu)^{\frac{n}{4}}\right)$
maximal induced bipartite subgraphs.
\end{enumerate}
\end{corollary}
All proofs are given in the next section.
Variations and further directions are discussed in a conclusion.

\section{Proofs}

First we establish our statement concerning the extremal graphs for \eqref{e2}.
For this we closely follow Eppstein's \cite{ep} argument.

\begin{proof}[Proof of Theorem \ref{theorem2}]
The proof is by induction on $n$.
For $n=0$, we have
${\rm mis}_{\leq k}(G)\leq 1\leq 3^{4k}4^{-3k}$ 
with equality throughout if and only if $k=0$,
in which case $G$ is indeed the disjoint union of $k$ copies of $K_3$ or $K_4$.
Henceforth, we assume that $n>0$.

First, suppose that $G$ has some vertex $u$ of degree at least $4$.
Note that every maximal independent set of $G$ is either a maximal independent set of $G-u$ or the union of $\{u\}$ and a maximal independent set of $G-N_G[u]$. Therefore, by induction, we obtain
\begin{align*}
{\rm mis}_{\leq k}(G)
&\leq {\rm mis}_{\leq k}(G-u)+{\rm mis}_{\leq k-1}(G-N_G[u])\\
&\leq 3^{4k-(n-1)}4^{(n-1)-3k}+3^{4(k-1)-(n-5)}4^{(n-5)-3(k-1)}\\
&<3^{4k-n}4^{n-3k}.
\end{align*}
In particular, equality in \eqref{e2} is impossible in this case.

Next, suppose that $G$ has maximum degree $3$.
Let $u$ be some vertex of degree $3$ in $G$.
Again, by induction, we obtain
\begin{align*}
{\rm mis}_{\leq k}(G)
&\leq {\rm mis}_{\leq k}(G-u)+{\rm mis}_{\leq k-1}(G-N_G[u])\\
&\leq 3^{4k-(n-1)}4^{(n-1)-3k}+3^{4(k-1)-(n-4)}4^{(n-4)-3(k-1)}\\
&=3^{4k-n}4^{n-3k},
\end{align*}
with equality throughout if and only if 
\begin{eqnarray}
&&\mbox{\it $G-u$ is the disjoint union of $k$ copies of $K_3$ or $K_4$ and}\label{ee1}\\
&&\mbox{\it $G-N_G[u]$ is the disjoint union of $k-1$ copies of $K_3$ or $K_4$.}\label{ee2}
\end{eqnarray}
If $G$ has a vertex $u$ of degree $3$ such that \eqref{ee1} or \eqref{ee2} fails,
then this implies ${\rm mis}_{\leq k}(G)<3^{4k-n}4^{n-3k}$.
Hence, we may assume that both \eqref{ee1} and \eqref{ee2} hold
for every vertex $u$ of degree $3$ in $G$.
Let $u$ be some vertex of degree $3$ in $G$ and 
let $K$ be a component of $G-u$ containing some neighbor $v$ of $u$.
Since $G$ has maximum degree $3$, 
it follows from \eqref{ee1} that $K$ is a copy of $K_3$
and that $v$ has degree $3$ in $G$.
If $N_G(u)\not=V(K)$, then $G-v$ is not the disjoint union of $k$ 
copies of $K_3$ or $K_4$, which is a contradiction.
It follows that $N_G(u)=V(K)$, which implies that $G$ is the disjoint union of $k$ 
copies of $K_3$ or $K_4$, as desired.

At this point, we may assume that $G$ has maximum degree at most $2$.
If $G$ has a vertex of degree $1$, then 
Eppstein shows 
${\rm mis}_{\leq k}(G)\leq \frac{8}{9}3^{4k-n}4^{n-3k}$
using very similar arguments as above.
If $G$ has an isolated vertex, 
then he shows 
${\rm mis}_{\leq k}(G)\leq \frac{16}{27}3^{4k-n}4^{n-3k}$.
If $G$ contains a cycle of length at least $4$, 
then he shows 
${\rm mis}_{\leq k}(G)\leq \frac{11}{12}3^{4k-n}4^{n-3k}$.
In all these cases, it follows that \eqref{e2} holds 
and that equality in \eqref{e2} is not possible.
The only remaining case is that $G$ is the disjoint union of triangles,
in which case the desired statement is trivial.
\end{proof}

\begin{proof}[Proof of Theorem \ref{theorem1}]
We do not give the explicit values of $\epsilon$ and $\delta$, 
but simply assume they are small enough so that all upcoming inequalities hold.
Let $G$, $n$, and $k$ be as in the statement.
Clearly, we may assume that some maximal independent set $I_0$ of $G$ 
has order exactly $k$.
Since $G$ has maximum degree at most $3$, this implies $k\geq \frac{n}{4}$.
Let $J_0=V(G)\setminus I_0$.
Let $I_1$ be the set of vertices in $I_0$ 
with at most $2$ neighbors in $J_0$.
Let $J_1$ be the set of vertices in $J_0$ 
with a neighbor in $I_1$.
Let $J_2$ be the set of vertices in $J_0\setminus J_1$
with at least $2$ neighbors in $I_0$.
Let $I_2$ be the set of vertices in $I_0$ 
with a neighbor in $J_2$.
Note that $|J_1|\leq 2|I_1|$,
$|I_2|\leq 3|J_2|$, and 
$I_1$ and $I_2$ are disjoint.
See Figure \ref{fig1} for an illustration.

Since $I_0$ is a maximal independent set,
every vertex in $J_0$ has a neighbor in $I_0$.
It follows that 
the number of edges of $G$ between $I_0$ and $J_0$ is 
\begin{itemize}
\item at least $|J_0|+|J_2|=n-k+|J_2|$
and
\item at most $3|I_0|-|I_1|=3k-|I_1|$,
\end{itemize}
which implies
\begin{eqnarray}\label{ea}
|I_1|+|J_2|&\leq & 4k-n\leq \epsilon n
\end{eqnarray}
and, hence,
\begin{eqnarray}\label{ec}
|J_1|+|J_2|&\leq &2|I_1|+|J_2|\stackrel{\eqref{ea}}{\leq} 2\epsilon n.
\end{eqnarray}
For $I_3=I_0\setminus (I_1\cup I_2)$ and $\ell=|I_3|$, 
we obtain
\begin{eqnarray}\label{eb}
\ell &=& |I_0|-|I_1|-|I_2|\geq k-|I_1|-3|J_2|\stackrel{\eqref{ea}}{\geq} k-3\epsilon n.
\end{eqnarray}
Let $I_3=\{ u_1,\ldots,u_\ell\}$.
By definition of $I_3$, 
every vertex $u_i$ in $I_3$ has exactly three neighbors in $J_0$, 
say $x_i$, $y_i$, and $z_i$,
and these three neighbors have $u_i$ as their unique neighbor in $I$. 
Recall that $I_0$ is an independent set, 
hence we have $N_G[u_i]=V_i$ for $V_i=\{ u_i,x_i,y_i,z_i\}$.
Since $G$ is $K_4$-free, 
we may assume that $x_i$ and $y_i$ are not adjacent for every $i\in [\ell]$.
See Figure \ref{fig1} for an illustration.

\begin{figure}[h]\centering
\unitlength 1mm 
\linethickness{0.4pt}
\ifx\plotpoint\undefined\newsavebox{\plotpoint}\fi 
\begin{picture}(100,51)(0,0)
\put(42,5){\circle*{1}}
\put(52,5){\circle*{1}}
\put(72,5){\circle*{1}}
\put(92,5){\circle*{1}}
\put(42,15){\circle*{1}}
\put(52,15){\circle*{1}}
\put(72,15){\circle*{1}}
\put(92,15){\circle*{1}}
\put(42,25){\circle*{1}}
\put(52,25){\circle*{1}}
\put(72,25){\circle*{1}}
\put(92,25){\circle*{1}}
\put(42,28){\makebox(0,0)[cc]{$z_1$}}
\put(52,28){\makebox(0,0)[cc]{$z_2$}}
\put(72,28){\makebox(0,0)[cc]{$z_i$}}
\put(92,28){\makebox(0,0)[cc]{$z_\ell$}}
\put(42,18){\makebox(0,0)[cc]{$y_1$}}
\put(52,18){\makebox(0,0)[cc]{$y_2$}}
\put(72,18){\makebox(0,0)[cc]{$y_i$}}
\put(92,18){\makebox(0,0)[cc]{$y_\ell$}}
\put(42,2){\makebox(0,0)[cc]{$x_1$}}
\put(52,2){\makebox(0,0)[cc]{$x_2$}}
\put(72,2){\makebox(0,0)[cc]{$x_i$}}
\put(92,2){\makebox(0,0)[cc]{$x_\ell$}}
\put(39,0){\framebox(6,30)[cc]{}}
\put(49,0){\framebox(6,30)[cc]{}}
\put(69,0){\framebox(6,30)[cc]{}}
\put(89,0){\framebox(6,30)[cc]{}}
\put(42,45){\circle*{1}}
\put(52,45){\circle*{1}}
\put(72,45){\circle*{1}}
\put(92,45){\circle*{1}}
\put(42,48){\makebox(0,0)[cc]{$u_1$}}
\put(52,48){\makebox(0,0)[cc]{$u_2$}}
\put(72,48){\makebox(0,0)[cc]{$u_i$}}
\put(92,48){\makebox(0,0)[cc]{$u_\ell$}}
\put(42,45){\line(0,-1){15}}
\put(52,45){\line(0,-1){15}}
\put(72,45){\line(0,-1){15}}
\put(92,45){\line(0,-1){15}}
\put(62,15){\makebox(0,0)[cc]{$\ldots$}}
\put(82,15){\makebox(0,0)[cc]{$\ldots$}}
\put(62,45){\makebox(0,0)[cc]{$\ldots$}}
\put(82,45){\makebox(0,0)[cc]{$\ldots$}}
\put(41.93,14.93){\line(0,-1){.9091}}
\put(41.93,13.112){\line(0,-1){.9091}}
\put(41.93,11.293){\line(0,-1){.9091}}
\put(41.93,9.475){\line(0,-1){.9091}}
\put(41.93,7.657){\line(0,-1){.9091}}
\put(41.93,5.839){\line(0,-1){.9091}}
\put(51.93,14.93){\line(0,-1){.9091}}
\put(51.93,13.112){\line(0,-1){.9091}}
\put(51.93,11.293){\line(0,-1){.9091}}
\put(51.93,9.475){\line(0,-1){.9091}}
\put(51.93,7.657){\line(0,-1){.9091}}
\put(51.93,5.839){\line(0,-1){.9091}}
\put(71.93,14.93){\line(0,-1){.9091}}
\put(71.93,13.112){\line(0,-1){.9091}}
\put(71.93,11.293){\line(0,-1){.9091}}
\put(71.93,9.475){\line(0,-1){.9091}}
\put(71.93,7.657){\line(0,-1){.9091}}
\put(71.93,5.839){\line(0,-1){.9091}}
\put(91.93,14.93){\line(0,-1){.9091}}
\put(91.93,13.112){\line(0,-1){.9091}}
\put(91.93,11.293){\line(0,-1){.9091}}
\put(91.93,9.475){\line(0,-1){.9091}}
\put(91.93,7.657){\line(0,-1){.9091}}
\put(91.93,5.839){\line(0,-1){.9091}}
\put(100,35){\line(-1,0){100}}
\put(35,0){\line(0,1){51}}
\put(23,37){\framebox(10,14)[cc]{}}
\put(23,16){\framebox(10,14)[cc]{}}
\put(11,37){\framebox(10,14)[cc]{}}
\put(11,16){\framebox(10,14)[cc]{}}
\put(16,40){\circle*{1}}
\put(28,27){\circle*{1}}
\put(18,28){\line(-1,6){2}}
\put(16,40){\line(-1,-6){2}}
\put(30,39){\line(-1,-6){2}}
\put(28,27){\line(-1,6){2}}
\put(16,25){\makebox(0,0)[cc]{$\underbrace{}_{\leq 2}$}}
\put(28,42){\makebox(0,0)[cc]{$\overbrace{}^{\geq 2}$}}
\put(16,48){\makebox(0,0)[cc]{$I_1$}}
\put(28,48){\makebox(0,0)[cc]{$I_2$}}
\put(16,19){\makebox(0,0)[cc]{$J_1$}}
\put(28,19){\makebox(0,0)[cc]{$J_2$}}
\put(6,37){\makebox(0,0)[cc]{$I_0$}}
\put(6,32){\makebox(0,0)[cc]{$J_0$}}
\put(37,37){\makebox(0,0)[cc]{$I_3$}}
\end{picture}
\caption{The structure of $G$ induced by $I_0$.}\label{fig1}
\end{figure}
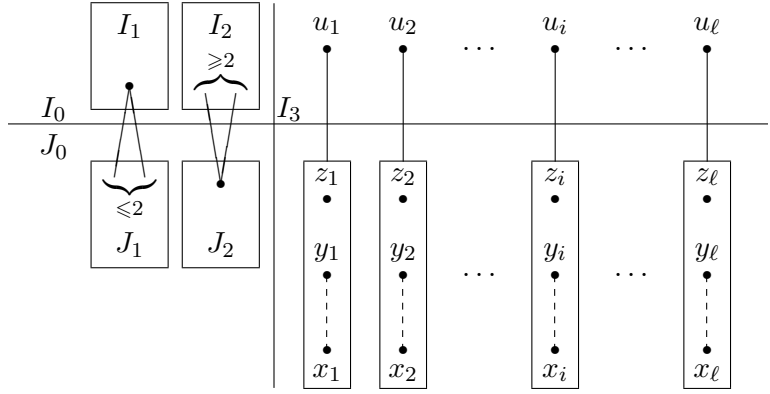
Let ${\cal I}$ be the set of all maximal independent sets of $G$ of order exactly $k$.
For every $i\in [\ell]$ and $I\in \mathcal{I}$, we have $|I\cap V_i|\geq 1$, 
for otherwise $I\cup \{u_i\}$ would be an independent set, 
contradicting the maximality of $I$.
For every $S\subseteq [\ell]$, 
let ${\cal I}_S$ be the set of all sets $I$ in ${\cal I}$ with 
$S=\{ i\in [\ell]:|I\cap V_i|\geq 2\}.$
If ${\cal I}_S$ is non-empty, then $k\geq \ell+|S|\stackrel{\eqref{eb}}{\geq} k-3\epsilon n+|S|$, which implies 
\begin{eqnarray}\label{ed}
|S|&\leq & 3\epsilon n.
\end{eqnarray}
Now, let $S\subseteq [\ell]$ be such that $|S|\leq 3\epsilon n$.
Let $I_4=\{ u_i:i\in [\ell]\setminus S\}$.
Clearly, we have 
\begin{eqnarray}\label{ee}
|I_4|&\geq & |I_3|-|S|
=\ell-|S|
\stackrel{\eqref{eb},\eqref{ed}}{\geq} k-3\epsilon n-3\epsilon n
= k-6\epsilon n.
\end{eqnarray}
Let $I_5$ be the set of all vertices $u_i$ in $I_4$ 
such that neither $x_i$ nor $y_i$ have a neighbor 
outside of 
$$U=\bigcup_{u_j\in I_4}V_j=\bigcup_{j\in [\ell]\setminus S}V_j.$$
For every $i\in S$, there are at most $6$ edges between $V_i$ and $U$.
For every $v\in J_1$, 
there are at most two edges between $v$ and $U$.
Finally,
for every $v\in J_2$, 
there is at most one edge between $v$ and $U$.
Therefore, we obtain
\begin{eqnarray}\label{ef}
|I_5|&\geq &|I_4|-6|S|-2|J_1|-|J_2|
\stackrel{\eqref{ec},\eqref{ed},\eqref{ee}}{\geq} k-6\epsilon n-18\epsilon n-4\epsilon n
=k-28\epsilon n.
\end{eqnarray}
Let $H$ be the auxiliary graph whose vertices are the sets $V_i$ with $u_i\in I_5$
and in which $V_i$ and $V_j$ are adjacent 
if there is some edge in $G$ between $V_i$ and $V_j$.
By construction, the maximum degree of $H$ is at most $6$.
Considering a maximal independent set in the square of $H$ 
yields a subset $I_6$ of $I_5$ with 
\begin{eqnarray}\label{eg}
|I_6|\geq \frac{|I_5|}{1+6+6\cdot 5}
\stackrel{\eqref{ef}}{\geq} \frac{k-28\epsilon n}{37}
\end{eqnarray}
such that, 
for every two distinct vertices $u_i$ and $u_j$ in $I_6$,
we have $N_G[v]\cap N_G[w]=\emptyset$ 
for every $v\in \{ x_i,y_i\}$ and every $w\in \{ x_j,y_j\}$.

Consider the partition ${\cal P}$ of $U$ given by the sets $V_i$ with $u_i\in I_4$.
A {\it transversal} of this partition ${\cal P}$ is a subset $T$ of $U$
intersecting each $V_i$ with $u_i\in I_4$ in exactly one vertex,
and such a transversal $T$ is {\it good} if, 
for every $u_i\in I_4$ such that $T$ contains $v\in \{ x_i,y_i\}$, 
the set $T$ also contains a neighbor of 
the unique vertex $v'$ in $\{ x_i,y_i\}\setminus \{ v\}$.
Note that, 
for every maximal independent set $I$ in ${\cal I}_S$,
the set $I\cap U$ is a good transversal of ${\cal P}$.
Therefore, 
if ${\cal T}$ is the set of all good transversals of ${\cal P}$,
then 
\begin{eqnarray}\label{eh}
|{\cal I}_S|&\leq &2^{V\setminus U}\cdot |{\cal T}|
\leq 2^{17\epsilon n}\cdot |{\cal T}|
=4^{34\epsilon \frac{n}{4}}\cdot |{\cal T}|.
\end{eqnarray}
Now, we consider the uniform distribution 
on the set of all $4^{|I_4|}$ transversals of ${\cal P}$
induced by selecting, for each vertex $u_i$ in $I_4$,
exactly one of the $4$ vertices of $V_i$ uniformly and independently at random.
Consider a random transversal $T$ of ${\cal P}$ that is selected in this way.
For each vertex $u_i\in I_5$, 
the probability that $T$ contains $v\in \{ x_i,y_i\}$ 
but no neighbor of the unique vertex $v'$ 
in $\{ x_i,y_i\}\setminus \{ v\}$ 
is at least, see Figure \ref{fig2}, 
\begin{eqnarray}\label{eq}
2\min\left\{ \frac{1}{4},\frac{1}{4}\cdot\frac{3}{4},\frac{1}{4}\cdot\frac{1}{2},\frac{1}{4}\cdot\frac{3}{4}\cdot\frac{3}{4}\right\}\geq \frac{1}{4}.
\end{eqnarray}

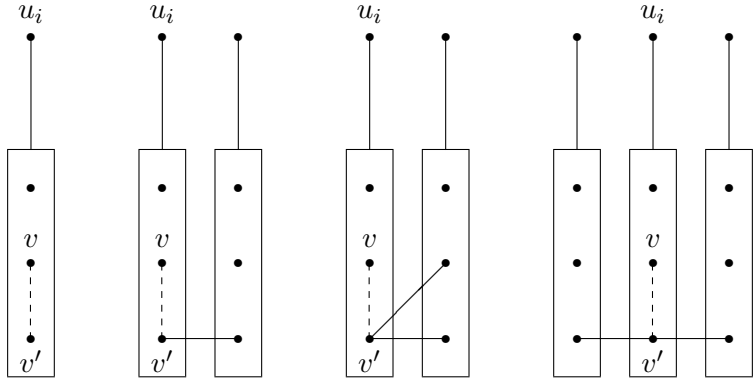
\begin{figure}[h]
\begin{center}
\unitlength 1mm 
\linethickness{0.4pt}
\ifx\plotpoint\undefined\newsavebox{\plotpoint}\fi 
\begin{picture}(6,48)(0,0)
\put(3,5){\circle*{1}}
\put(3,15){\circle*{1}}
\put(3,25){\circle*{1}}
\put(3,18){\makebox(0,0)[cc]{$v$}}
\put(3,2){\makebox(0,0)[cc]{$v'$}}
\put(0,0){\framebox(6,30)[cc]{}}
\put(3,45){\circle*{1}}
\put(3,48){\makebox(0,0)[cc]{$u_i$}}
\put(3,45){\line(0,-1){15}}
\put(2.93,14.93){\line(0,-1){.9091}}
\put(2.93,13.112){\line(0,-1){.9091}}
\put(2.93,11.293){\line(0,-1){.9091}}
\put(2.93,9.475){\line(0,-1){.9091}}
\put(2.93,7.657){\line(0,-1){.9091}}
\put(2.93,5.839){\line(0,-1){.9091}}
\end{picture}\hspace{1cm}
\unitlength 1mm 
\linethickness{0.4pt}
\ifx\plotpoint\undefined\newsavebox{\plotpoint}\fi 
\begin{picture}(16,48)(0,0)
\put(13,5){\circle*{1}}
\put(3,5){\circle*{1}}
\put(13,15){\circle*{1}}
\put(3,15){\circle*{1}}
\put(13,25){\circle*{1}}
\put(3,25){\circle*{1}}
\put(3,18){\makebox(0,0)[cc]{$v$}}
\put(3,2){\makebox(0,0)[cc]{$v'$}}
\put(10,0){\framebox(6,30)[cc]{}}
\put(0,0){\framebox(6,30)[cc]{}}
\put(13,45){\circle*{1}}
\put(3,45){\circle*{1}}
\put(3,48){\makebox(0,0)[cc]{$u_i$}}
\put(13,45){\line(0,-1){15}}
\put(3,45){\line(0,-1){15}}
\put(2.93,14.93){\line(0,-1){.9091}}
\put(2.93,13.112){\line(0,-1){.9091}}
\put(2.93,11.293){\line(0,-1){.9091}}
\put(2.93,9.475){\line(0,-1){.9091}}
\put(2.93,7.657){\line(0,-1){.9091}}
\put(2.93,5.839){\line(0,-1){.9091}}
\put(3,5){\line(1,0){10}}
\end{picture}\hspace{1cm}
\unitlength 1mm 
\linethickness{0.4pt}
\ifx\plotpoint\undefined\newsavebox{\plotpoint}\fi 
\begin{picture}(16,48)(0,0)
\put(13,5){\circle*{1}}
\put(3,5){\circle*{1}}
\put(13,15){\circle*{1}}
\put(3,15){\circle*{1}}
\put(13,25){\circle*{1}}
\put(3,25){\circle*{1}}
\put(3,18){\makebox(0,0)[cc]{$v$}}
\put(3,2){\makebox(0,0)[cc]{$v'$}}
\put(10,0){\framebox(6,30)[cc]{}}
\put(0,0){\framebox(6,30)[cc]{}}
\put(13,45){\circle*{1}}
\put(3,45){\circle*{1}}
\put(3,48){\makebox(0,0)[cc]{$u_i$}}
\put(13,45){\line(0,-1){15}}
\put(3,45){\line(0,-1){15}}
\put(2.93,14.93){\line(0,-1){.9091}}
\put(2.93,13.112){\line(0,-1){.9091}}
\put(2.93,11.293){\line(0,-1){.9091}}
\put(2.93,9.475){\line(0,-1){.9091}}
\put(2.93,7.657){\line(0,-1){.9091}}
\put(2.93,5.839){\line(0,-1){.9091}}
\put(3,5){\line(1,0){10}}
\put(3,5){\line(1,1){10}}
\end{picture}\hspace{1cm}
\unitlength 1mm 
\linethickness{0.4pt}
\ifx\plotpoint\undefined\newsavebox{\plotpoint}\fi 
\begin{picture}(26,48)(0,0)
\put(3,5){\circle*{1}}
\put(23,5){\circle*{1}}
\put(13,5){\circle*{1}}
\put(3,15){\circle*{1}}
\put(23,15){\circle*{1}}
\put(13,15){\circle*{1}}
\put(3,25){\circle*{1}}
\put(23,25){\circle*{1}}
\put(13,25){\circle*{1}}
\put(13,18){\makebox(0,0)[cc]{$v$}}
\put(13,2){\makebox(0,0)[cc]{$v'$}}
\put(0,0){\framebox(6,30)[cc]{}}
\put(20,0){\framebox(6,30)[cc]{}}
\put(10,0){\framebox(6,30)[cc]{}}
\put(3,45){\circle*{1}}
\put(23,45){\circle*{1}}
\put(13,45){\circle*{1}}
\put(13,48){\makebox(0,0)[cc]{$u_i$}}
\put(3,45){\line(0,-1){15}}
\put(23,45){\line(0,-1){15}}
\put(13,45){\line(0,-1){15}}
\put(12.93,14.93){\line(0,-1){.9091}}
\put(12.93,13.112){\line(0,-1){.9091}}
\put(12.93,11.293){\line(0,-1){.9091}}
\put(12.93,9.475){\line(0,-1){.9091}}
\put(12.93,7.657){\line(0,-1){.9091}}
\put(12.93,5.839){\line(0,-1){.9091}}
\put(13,5){\line(-1,0){10}}
\put(13,5){\line(1,0){10}}
\end{picture}
\end{center}
\caption{Let $\{ v,v'\}=\{ x_i,y_i\}$.
The probability $q$ that $T$ contains $v$ but no neighbor of $v'$
depends on the number $d$ of neighbors of $v'$ outside of $V_i$
and their distribution.
In the figure we illustrate --- from left to right ---- 
the four possible cases 
$d=0$, 
$d=1$,
$d=2$ and both neighbors of $v'$ outside of $V_i$ are in the same $V_j$,
and 
$d=2$ and the two neighbors of $v'$ outside of $V_i$ 
are in distinct $V_j$'s.
In the first case, we have $d=0$, 
$v'$ has all its neighbors in $V_i$,
$v$ lies in $T$ with probability $1/4$,
and, hence, $q=\frac{1}{4}$.
In the second case, we have $d=1$,
$v$ lies in $T$ with probability $1/4$,
the unique neighbor of $v'$ outside of $V_i$ lies outside of $T$
with probability  $3/4$, and, by the independence of these events,
$q=\frac{1}{4}\cdot \frac{1}{4}$.
In the third case, 
$v$ lies in $T$ with probability $1/4$,
 $T$ contains none of the two neighbors of $v'$ in $V_j$
with probability $1/2$, and, by the independence of these events,
$q=\frac{1}{4}\cdot \frac{1}{2}$.
Finally, 
in the fourth case, 
$v$ lies in $T$ with probability $1/4$,
each of the two neighbors of $v'$ outside of $V_i$
lie outside of $T$ with probability $3/4$, and,
by the independence of these events,
$q=\frac{1}{4}\cdot \frac{3}{4}\cdot \frac{3}{4}$.
The initial factor $2$ in \eqref{eq} reflects that 
$v$ may be $x_i$ or $y_i$.}\label{fig2}
\end{figure}
By the construction of $I_6$,
it follows that the probability $p$ that $T$ is good satisfies
\begin{eqnarray}\label{ei}
p\leq \left(1-\frac{1}{4}\right)^{|I_6|}
\stackrel{\eqref{eg}}{\leq} 
\left(\frac{3}{4}\right)^{\frac{k-28\epsilon n}{37}}
\stackrel{k\geq n/4}{\leq} 
4^{-\left(1-\frac{\log_2(3)}{2}\right)\frac{n/4-28\epsilon n}{37}}
=4^{-\left(1-\frac{\log_2(3)}{2}\right)\frac{(1-112\epsilon)}{37}\frac{n}{4}},
\end{eqnarray}
which implies that 
\begin{eqnarray}\label{ej}
|{\cal T}|
&\leq & p\cdot 4^{|I_4|}
\stackrel{\eqref{ei},|I_4|\leq (1+\epsilon)\frac{n}{4}}{\leq} 
4^{-\left(1-\frac{\log_2(3)}{2}\right)\frac{(1-112\epsilon)}{37}\frac{n}{4}}
\cdot 4^{(1+\epsilon)\frac{n}{4}}
=
4^{\left(1+\epsilon-\left(1-\frac{\log_2(3)}{2}\right)\frac{(1-112\epsilon)}{37}\right)\frac{n}{4}}
\end{eqnarray}
and, hence,
\begin{eqnarray}\label{ek}
|{\cal I}_S|
&\stackrel{\eqref{eh}}{\leq}  & 4^{34\epsilon \frac{n}{4}}\cdot |{\cal T}|
\stackrel{\eqref{ej}}{\leq} 
4^{\left(1+35\epsilon-\left(1-\frac{\log_2(3)}{2}\right)\frac{(1-112\epsilon)}{37}\right)\frac{n}{4}}.
\end{eqnarray}
Now, we obtain
\begin{eqnarray*}
|{\cal I}| 
&\stackrel{\eqref{ed}}{=}& \sum\limits_{S\subseteq [\ell]:|S|\leq 3\epsilon n}|{\cal I}_S|\\
&\stackrel{\eqref{ek}}{\leq} &
4^{\left(1+35\epsilon-\left(1-\frac{\log_2(3)}{2}\right)\frac{(1-112\epsilon)}{37}\right)\frac{n}{4}}\cdot\sum\limits_{s=0}^{\lfloor 3\epsilon n\rfloor}{|I_3|\choose s}\\
& \stackrel{|I_3|\leq (1+\epsilon)\frac{n}{4}}{\leq} &
4^{\left(1+35\epsilon-\left(1-\frac{\log_2(3)}{2}\right)\frac{(1-112\epsilon)}{37}\right)\frac{n}{4}}
\cdot 2^{h\left(\frac{12\epsilon}{1+\epsilon}\right)(1+\epsilon)\frac{n}{4}}\\
&=& 
4^{\left(1+h\left(\frac{12\epsilon}{1+\epsilon}\right)\frac{(1+\epsilon)}{2}+35\epsilon-\left(1-\frac{\log_2(3)}{2}\right)\frac{(1-112\epsilon)}{37}\right)\frac{n}{4}},
\end{eqnarray*}
where $h(\alpha)=-\alpha\log_2(\alpha)-(1-\alpha)\log_2(1-\alpha)$
and we use the well known estimate
$\sum\limits_{s=0}^{\lfloor \alpha N\rfloor}{N\choose s}
\leq 2^{h(\alpha)N}$ for $0<\alpha\leq \frac{1}{2}$.
Since
$$
\lim\limits_{\epsilon\to 0^+}\left(1+h\left(\frac{12\epsilon}{1+\epsilon}\right)\frac{(1+\epsilon)}{2}+35\epsilon-\left(1-\frac{\log_2(3)}{2}\right)\frac{(1-112\epsilon)}{37}\right)=
1-\frac{1}{37}\left(1-\frac{\log_2(3)}{2}\right)<1,$$
the desired statement follows.
\end{proof}

\begin{proof}[Proof of Corollary \ref{corollary1}]
Let $G$, $n$, and $k$ be as in the statement.
At some points within the proof, 
we use estimates that require $\eta$ be to sufficiently small.
The proof is by induction on $n$.
For $n=0$, we have $k=0$ nd the statement is trivial.
Now, let $n>0$.
If $G$ has maximum degree $d\geq 4\geq 4-\eta$, 
then, let the vertex $u$ be of maximum degree.
By induction (I), we obtain
\begin{eqnarray*}
{\rm mis}_k(G) & \leq & 
{\rm mis}_k(G-u)
+{\rm mis}_{k-1}(G-N_G[u])\\
&\stackrel{(I)}{\leq}&
(4-\eta)^{(5-\eta)k-(n-1)}(5-\eta)^{(n-1)-(4-\eta)k}\\
&&+(4-\eta)^{(5-\eta)(k-1)-(n-(d+1))}(5-\eta)^{(n-(d+1))-(4-\eta)(k-1)}\\
&\stackrel{\eta\geq 0}{\leq}&
(4-\eta)^{(5-\eta)k-(n-1)}(5-\eta)^{(n-1)-(4-\eta)k}\\
&&+(4-\eta)^{(5-\eta)(k-1)-(n-(5-\eta))}(5-\eta)^{(n-(5-\eta))-(4-\eta)(k-1)}\\
&=& 
(4-\eta)^{(5-\eta)k-n}(5-\eta)^{n-(4-\eta)k}.
\end{eqnarray*}
Note that here we only require $\eta\geq 0$.

\begin{figure}[h]
    \centering
\begin{tikzpicture}[scale=0.8]
\begin{axis}[
y label style={at={(-0.05,0.5)}},
    samples=100,domain=0.2:0.333]
    \addplot[thin, dotted, mark=none ]plot (\x,\x*ln(1/\x);
    \addplot [thin] table {
0.2 0.32188
0.25 0.3465
0.333 0.366
};
\addplot [thin] table {
0.25 0.343
0.27 0.343
};
\addplot [dotted] table {
0.25 0.343
0.25 0.3465
};
\addplot [dotted] table {
0.27 0.343
0.27 0.3511
};
\addplot [dashed] table {
0.2173 0.331
0.277 0.355
};
\end{axis}
\end{tikzpicture}
\caption{
The plot shows upper bounds on $\frac{\ln({\rm mis}_k(G))}{n}$ 
as a function of $x=\frac{k}{n}$.
As in Figure \ref{fig0}, two of the staight line segments 
correspond
to Eppstein's bound \eqref{e2} in $\left[\frac{1}{4},\frac{1}{3}\right]$ 
and 
to Nielsen's bound in $\left[\frac{1}{5},\frac{1}{4}\right]$.
The horizontal line segments illustrates a bound as in \eqref{e4} 
from Theorem \ref{theorem1} 
in $\left[\frac{1}{4},\frac{1+\epsilon}{4}\right]$.
Continuously increasing $\eta$ from $0$ 
(Nielsen's bound)
towards $1$
(Eppstein's bound)
yields the dashed line segment corresponding to \eqref{e5} 
as in Corollary \ref{corollary1}.
In fact, the plot illustrates $\eta=0.4$.}
\label{figcor1}
\end{figure}
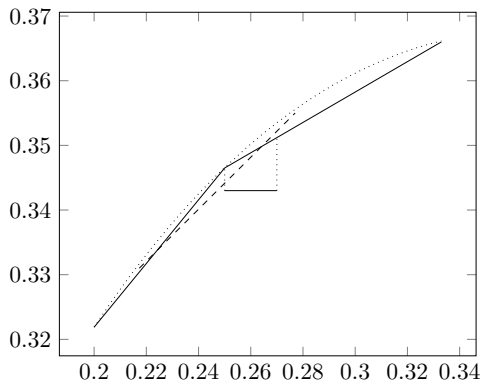
Hence, we may assume that $G$ has maximum degree at most $3$.
Fix $\epsilon, \delta>0$ as in Theorem \ref{theorem1}.
See Figure \ref{figcor1} for an illustration.
If $k \leq (1+\epsilon)\frac{n}{4}$,
then \eqref{e4} implies \eqref{e5} provided that $\eta>0$
is sufficiently small in terms of $\epsilon$ and $\delta$.
For $k>(1+\epsilon)\frac{n}{4}$, 
then \eqref{e2} implies \eqref{e5} 
provided that $\eta>0$
is sufficiently small in terms of $\epsilon$.
This completes the proof.
\end{proof}

\begin{proof}[Proof of Corollary \ref{corollary2}]
(i) Let $G$ and $n$ be as in the statement.
Let $G$ have ${\rm mibs}$ maximal induced bipartite subgraphs.
Counting these subgraphs as explained before \eqref{e3} 
and using implies \eqref{e1}, \eqref{e2}, and \eqref{e5}
implies ${\rm mibs} \leq {\rm mibs}_1+{\rm mibs}_2$, where
\begin{eqnarray*}
{\rm mibs}_1 &=& \sum\limits_{k=0}^{p}
(4-\eta)^{(5-\eta)k-n}(5-\eta)^{n-(4-\eta)k}3^{4k-(n-k)}4^{(n-k)-3k}\mbox{ and }\\
{\rm mibs}_2 &=& \sum\limits_{k=p+1}^{n}3^{4k-n}4^{n-3k}3^{\frac{n-k}{3}}
\end{eqnarray*}
for every integer $p$ with $0\leq p\leq n$.

Since 
$(5-\eta)\ln(4-\eta)-(4-\eta)\ln(5-\eta)+5\ln(3)-4\ln(4)>0$ for $\eta\in [0,1]$,
the terms in the sum ${\rm mibs}_1$ are increasing in $k$.
Since $4\ln(3)-3\ln(4)-\frac{\ln\left(3\right)}{3}<0$,
the terms in the sum ${\rm mibs}_2$ are decreasing in $k$.
For $\eta=0$ and $p=\left\lfloor\frac{n}{4}\right\rfloor$,
all terms in ${\rm mibs}_1$ and ${\rm mibs}_2$ are
bounded by $12^{\frac{n}{4}}$.
Now, let $\eta>0$ as in Corollary \ref{corollary1}.
Choosing 
$p=\left\lfloor\frac{(1+\xi)n}{4}\right\rfloor$
for some sufficiently small $\xi>0$ yields that
${\rm mibs}_1$ and ${\rm mibs}_2$
are $O\left((12-\nu)^{\frac{n}{4}}\right)$
for some $\nu>0$,
which completes the proof of (i).

\bigskip

\noindent (ii) Let $G$ and $n$ be as in the statement.
The proof is by induction on $n$.
For $n=0$, the statement is trivial.
Now, let $n>0$.
If $G$ is $K_4$-free, then (i) implies the statement.
If $G$ contains a copy $K$ of $K_4$, 
then the maximum degree condition implies that $K$ is a component of $G$. 
Observe that every maximal induced bipartite subgraphs of $G$ 
contains precisely two vertices of $K$.
Since $6\cdot (12-\nu)^{\frac{n-4}{4}}<(12-\nu)^{\frac{n}{4}}$,
the statement follows by induction.
\end{proof}

\section{Conclusion}

Our initial motivation led us to the case $k=\frac{n}{4}$.
Nevertheless, simple adaptions within the proof of Theorem \ref{theorem1} 
easily allow to show:
\begin{quote}
{\it Let $s$ be a positive integer at least $4$. There exist $\epsilon_s, \delta_s>0$ with the following property.
If $G$ is a $K_s$-free graph of order $n$ and maximum degree at most $s-1$,
and $k\leq (1+\epsilon_s)\frac{n}{s}$ is a positive integer,
then ${\rm mis}_k(G) \leq (s-\delta_s)^{\frac{n}{s}}$.}
\end{quote}
Again this improves Nielsen's bound \cite{ni} for $K_s$-free graphs 
and $k$ close to $\frac{n}{s}$.

The number of maximal independent sets in triangle-free graphs 
has been studied by Hujter and Tuza \cite{hutu}. 
Palmer and Patk\'os \cite{papa} showed a common generalization 
of the bounds of Moon \& Moser \cite{momo,mimu}
and Hujter \& Tuza \cite{hutu}.
In this context it seems interesting 
to investigate ${\rm mis}_k(G)$ 
--- and not just ${\rm mis}(G)$ ---
for triangle-free graphs $G$ for the full range of values of $k$.

\paragraph{Acknowledgment} 
The first and third authors were partially supported by
the Deutsche Forschungsgemeinschaft 
(DFG, German Research Foundation) -- project number 545935699.
The second author is supported by  JST ASPIRE, grant number JPMJAP2302.

\end{document}